\documentclass{amsart}

\usepackage{amsmath}
\usepackage{amsthm}
\usepackage{amssymb}
\usepackage{xspace}
\usepackage{graphicx}
\usepackage{color}

\theoremstyle{plain}

\newtheorem{theorem}{Theorem}
\newtheorem{lemma}[theorem]{Lemma}

\newtheorem{proposition}[theorem]{Proposition}
\newtheorem*{lemma2.3}{Lemma 2.3 of \cite{hochman1}}

\theoremstyle{definition}

\newtheorem{example}{Example}

\newcommand{\comment}[1]{}
\newcommand{\rea}{\ensuremath{\mathbf{R}}\xspace}

\newcommand{\ints}{\ensuremath{\mathbf{Z}}\xspace}
\newcommand{\intd}[1]{\,\mathrm{d}#1 \,}
\newcommand{\floor}[1]{\lfloor #1 \rfloor}
\newcommand{\closure}[1]{\ensuremath{\overline{#1}}\xspace}
\newcommand{\udim}[2]{\overline{D}_{#1}(#2)}
\newcommand{\diffk}[1]{\ensuremath{\xdiff_+^{#1}(\rea)}\xspace}

\newcommand{\scdi}[3]{\ensuremath{\langle #1 \rangle_{#2, #3}^{\square}}\xspace}

\newcommand{\scfl}[3]{\ensuremath{{#1}_{#2, #3}^{\square}}\xspace}

\newcommand\restr[2]{{                                  
  \left.\kern-\nulldelimiterspace               	
  #1                                                    
  \vphantom{\big|}                                      
  \right|_{#2}                                          
  }}

\DeclareMathOperator{\supp}{supp}
\DeclareMathOperator{\xdiff}{diff}
\DeclareMathOperator{\sign}{sign}

\begin{document}

\title{On the asymptotics of the scenery flow}

\author{Magnus Aspenberg}
\address{Centre for Mathematical Sciences\\ Lund University\\ Box 118\\ 22 100 Lund\\ SWEDEN}
\email{magnusa@maths.lth.se}
\author{Fredrik Ekstr\"om}
\email{fredrike@maths.lth.se}
\author{Tomas Persson}
\email{tomasp@maths.lth.se}
\author{J\"org Schmeling}
\email{joerg@maths.lth.se}

\maketitle

\section{Introduction}
Various notions of ``zooming in'' on measures exist in the literature and the scenery flow
is one of them. A variant similar to the scenery flow was introduced by Bandt in \cite{bandt1} and studied
by Graf in \cite{graf1}. More recently, the scenery flow was defined by Gavish in \cite{gavish1}
and studied by Hochman in \cite{hochman1, hochman2}.
One crucial step is to describe the joint
asymptotics of the scenery flows generated by a measure and the
measure transported by a local diffeomorphism, see Lemma 2.3 from \cite{hochman1}
below. Unfortunately the lemma was stated without proof, and the
assumptions are not sufficient to guarantee the conclusion. In this
paper we give sufficient conditions for similar conclusions,
and provide complete proofs. These statements show that the main
results of \cite{hochman1} are not affected. We think however it is important on
its own to have some more detailed insight to this crucial point in the properties of the
scenery flow.

These results arose from discussions at the Lund dynamics seminar,
and email correspondence with Michael Hochman. Hochman pointed out
that he possesses a corrected statement similar to Proposition
\ref{fintedimproposition} (see Proposition 1.9 in \cite{hochman2}).

Let $I = [-1, \, 1]$ and for any subset $A$ of $I$ and any $t \geq 0$
let
	$$
	A_t = e^{-t}A.
	$$
Let $\mathcal{M}(\rea)$ be the set of Radon measures on \rea and let $\mathcal{M}^{\square}$
be the topological space of Borel probability measures on $I$, considered as a subspace of
the dual of $C(I)$ with the weak-$*$ topology. (With this topology $\mathcal{M}^{\square}$ is compact.)
For any $\mu \in \mathcal{M}(\rea)$, any $x \in \supp \mu$ and any $t \geq 0$,
let $\scfl{\mu}{x}{t}$ be the measure in $\mathcal{M}^{\square}$ defined by
	$$
	\scfl{\mu}{x}{t}(A) =
	\frac{\mu(x + A_t)}{\mu(x + I_t)}, \qquad A \in \mathcal{B}(I).
	$$
Thus $\scfl{\mu}{x}{t}$ tells what $\mu$ looks like in a window of radius $e^{-t}$ around $x$.
The one-parameter family $\{ \scfl{\mu}{x}{t} \}_{t \geq 0}$ is called the \emph{scenery} of $\mu$
at $x$.

A Borel probability measure on $\mathcal{M}^{\square}$ is called a \emph{distribution}.
The set of distributions is given the weak-$*$ topology that comes from considering it
as a subspace of the dual of $C(\mathcal{M}^{\square})$. For any measure $\mu \in \mathcal{M}(\rea)$,
any $x \in \supp \mu$ and any $T \geq 0$, let $\scdi{\mu}{x}{T}$ be the distribution defined
by
	$$
	\scdi{\mu}{x}{T} (E) = \frac{1}{T} \lambda (\{ t \leq T; \, \scfl{\mu}{x}{t} \in E \}), \qquad
	E \in \mathcal{B}(\mathcal{M}^{\square}),
	$$
where $\lambda$ is Lebesgue measure (one can show that $t \mapsto \scfl{\mu}{x}{t}$ is left continuous,
hence measurable, so that $\scdi{\mu}{x}{T}$ is well defined. See for example Lemma \ref{contlemma} below).
If $P$ is a distribution and
	$$
	\lim_{T \to \infty} \scdi{\mu}{x}{T} = P
	$$
then $\mu$ is said to \emph{generate} $P$ at $x$. Thus $\mu$ generates $P$ at $x$ if and
only if
	$$
	\lim_{T \to \infty} \frac{1}{T} \int_{0}^{T} g(\scfl{\mu}{x}{t}) \intd{t} =
	\int g \intd{P}
	$$
for every $g \in C(\mathcal{M}^{\square})$.

If $\mu$ is a measure on $X$ and $f: X \to Y$ is a measurable function, then $f\mu$ denotes the
measure on $Y$ given by
	$$
	(f \mu)(A) = \mu(f^{-1}(A))
	$$
for measurable subsets $A$ of $Y$.
The set of orientation preserving $C^k$-diffeomorphisms $\rea \to \rea$ will be denoted by
$\diffk{k}${}.

The following appears as Lemma 2.3 in \cite{hochman1}.
\begin{lemma2.3}
Let $\mu \in \mathcal{M}(\rea)$, $x \in \supp \mu$ and $f \in \diffk{1}$. Then after
a time-shift of $s = \log f'(x)$ the sceneries $\scfl{\mu}{x}{t}$ and $\scfl{(f\mu)}{f(x)}{t}$ are
asymptotic, i.e.
	$$
	\lim_{t \to \infty} (\scfl{\mu}{x}{t} - \scfl{(f\mu)}{f(x)}{t - s}) = 0 \quad (\text{weak-}*).
	$$
In particular $\mu$ generates $P$ at $x$ if and only if $f\mu$ generates $P$ at $f(x)$.
\end{lemma2.3}
The idea is that locally around $x$, the diffeomorphism $f$ acts as a translation that
takes $x$ to $f(x)$ composed with a scaling around $f(x)$ by the factor $f'(x) = e^s$.
However, no detailed proof is given in \cite{hochman1}, and in fact it is possible to
construct counterexamples to both of the statements in Lemma 2.3 (see Section \ref{examplessec}).
What can go wrong is that if $\mu$ has a lot of measure near the boundary
of $x + I_t$, then even the small distortion caused by $f$ can push a relatively
large amount of measure across the boundary, thus creating a big difference in
normalisation between the measures $\scfl{\mu}{x}{t}$ and $\scfl{(f\mu)}{f(x)}{t - s}$. Under
a certain (mild) condition, the set of $t \leq T$ for which this happens has small
relative density in $[0, \, T]$ if $T$ is large. Then the distributions $\scdi{\mu}{x}{T}$
and $\scdi{f\mu}{f(x)}{T}$ will be asymptotic, even if $\scfl{\mu}{x}{t}$ and
$\scfl{(f\mu)}{f(x)}{t - s}$ are not.

\section{An Auxiliary Result and its Consequences}
\begin{lemma} \label{TClemma}
Let $\mu \in \mathcal{M}(\rea)$ and $x \in \supp \mu$, and suppose that
        \begin{equation}        \label{TC}
        \lim_{K \to \infty} \limsup_{\gamma \to 1^+} \limsup_{T \to \infty}
        \frac{1}{T} \lambda \left(
        \left\{t \leq T; \, \frac{\mu(x + \gamma I_t)}{\mu(x + I_t)} \geq K \right\}
        \right) = 0.
        \end{equation}
Then $\scdi{\mu}{x}{T}$ and $\scdi{f\mu}{f(x)}{T}$ are asymptotic for any $f \in \diffk{1}$. 
\end{lemma}

Before we proceed to the proof of Lemma~\ref{TClemma}, we will state
some consequences. The proofs of Propositions \ref{fintedimproposition}, \ref{succeedproposition}
and \ref{failproposition} are given in Section \ref{applicationsection}.

The upper local dimension of a measure $\mu \in \mathcal{M}(\rea)$ at $x$ is
defined to be
	$$
	\udim{\mu}{x} = \limsup_{r \to 0^+} \frac{\log \mu([x - r, \, x + r])}{\log r}.
	$$
It can be shown that if a Borel set $E$ has positive $\mu$-measure and $\udim{\mu}{x} \geq c$
for all $x \in E$ then the packing dimension of $E$ is $\geq c$ (see for example Proposition 2.3
in \cite{falconer1}). This implies that $\udim{\mu}{x} \leq 1$ for $\mu$-a.e. $x$, so the
next proposition says in particular that $\mu$-a.e. $x$ is such that 
$\scdi{\mu}{x}{T}$ and $\scdi{f\mu}{f(x)}{T}$ are asymptotic for all $f \in \diffk{1}$.

\begin{proposition}	\label{fintedimproposition}
Let $\mu \in \mathcal{M}(\rea)$ and suppose that $\udim{\mu}{x} < \infty$.
Then $\scdi{x}{\mu}{T}$ and $\scdi{f\mu}{f(x)}{T}$ are asymptotic for all
$f \in \diffk{1}$.
\end{proposition}

If $\scdi{\mu}{x}{T}$ has an $\omega$-limit point $P$ such that $0 \in \supp \nu$ for
all $\nu \in \supp P$ then the subset of $\mathcal{M}^\square$ where
	$$
	S_t^{\square}: \nu \mapsto \nu_{0, t}
	$$
is defined for all $t \geq 0$ has full $P$-measure, and
$(\mathcal{M}^{\square}, S^{\square}, P)$ is a measure preserving flow%
\footnote{%
Proof sketch: If $\scdi{\mu}{x}{T_k} \to P$, it can be shown that
$P(\{ \nu; \, \nu({c}) > 0 \}) = 0$ for any $c \in (0, \, 1) \setminus \{0\}$,
and thus for fixed $t > 0$ the map $S_t^{\square}$ is continuous $P$-a.e.
(since $S_t^{\square}$ is discontinuous at $\nu$ if and only if $\nu(\{\pm e^{-t}\}) > 0$).
Hence for any $g \in C(\mathcal{M}^{\square})$, the composition $g \circ S_t^{\square}$
is continuous $P$-a.e., so
	$$
	\lim_{T \to \infty} \int g \circ S_t^{\square} \intd{\scdi{\mu}{x}{T}}
	= \int g \circ S_t^{\square} \intd{P}
	$$
(see for example Theorem 3.1.5 in \cite{stroock1}). But this implies that
$S_t^{\square} \scdi{\mu}{x}{T_k} \to S_t^{\square} P$, and it follows that $P$ is
$S^{\square}$-invariant since
$S_t^{\square} \scdi{\mu}{x}{T}$ and $\scdi{\mu}{x}{T}$ are asymptotic.
}%
. A consequence of the next proposition is
that in this situation, $\scdi{\mu}{x}{T}$ and $\scdi{f\mu}{f(x)}{T}$ are asymptotic
for all $f \in \diffk{1}$.

\begin{proposition}	\label{succeedproposition}
Let $\mu \in \mathcal{M}(\rea)$ be a measure with $x \in \supp \mu$, and suppose
that $P(\{ (1 - \beta)\delta_{-1} + \beta \delta_1; \, 0 \leq \beta \leq 1 \}) = 0$
for every distribution $P$ in the $\omega$-limit set of $\scdi{\mu}{x}{T}$.
Then $\scdi{\mu}{x}{T}$ and $\scdi{f\mu}{f(x)}{T}$ are asymptotic for all $f \in \diffk{1}$.
\end{proposition}

The following proposition is almost the converse of Proposition \ref{succeedproposition}.

\begin{proposition}	\label{failproposition}
Let $\mu \in \mathcal{M}(\rea)$ and $x \in \supp \mu$, and suppose that $\scdi{\mu}{x}{T}$
has an $\omega$-limit point $P$ such that
	$$
	P(\{ (1 - \beta)\delta_{-1} + \beta \delta_{1}; \, 0 < \beta < 1 \}) > 0.
	$$
Then there exists a diffeomorphism $f \in \diffk{1}$ such that $\scdi{\mu}{x}{T}$
and $\scdi{f\mu}{f(x)}{T}$ are not asymptotic.
\end{proposition}

\section{The Proof of Lemma \ref{TClemma}}
The proof of Lemma \ref{TClemma} goes via a few other lemmas.

\begin{lemma} \label{tubelemma}
Let $\mu \in \mathcal{M}(\rea)$ and let $\delta \geq 0$. Then
for $a \leq b$
        $$
        \int_a^b \mu([y - \delta, \, y + \delta]) \intd{y} \leq
        2 \delta \mu([a - \delta, \, b + \delta]).
        $$

\begin{proof}
Using Fubini's theorem
        \begin{align*} 
        \int_a^b &\mu([y - \delta, \, y + \delta]) \intd{y} =
        (\mu \times \lambda) (\{ (x, y); \, y - \delta \leq x \leq y + \delta, \, a \leq y \leq b \}) \leq \\
        &(\mu \times \lambda) (\{ (x, y); \, a - \delta \leq x \leq b + \delta,
                                                                   \, x - \delta \leq y \leq x + \delta \}) \stackrel{*}{=} \\
        &(\mu \times \lambda) (\{ (x, y); \, a - \delta \leq x \leq b + \delta,
                                                                   \, -\delta \leq y \leq \delta \}) = \\
        &2\delta \mu([a - \delta, \, b + \delta]),
        \end{align*}
using at $*$ that $\lambda$ is translation invariant.
\end{proof}
\end{lemma}

\begin{lemma}   \label{hdeltalemma}
Suppose that $\mu \in \mathcal{M}(\rea)$ has $0 \in \supp \mu$ and satisfies the condition \eqref{TC} at $x = 0$.
Let $A \subset I$ be an interval containing $0$ and define for $\delta > 0$
        $$
        h_{\delta}(t) = \frac{\mu(A_{t - \delta})}{\mu(I_{t + \delta})}
        - \frac{\mu(A_{t + \delta})}{\mu(I_{t - \delta})}.
        $$
Then for any $\alpha > 0$
        $$
        \lim_{\delta \to 0} \limsup_{T \to \infty} \frac{1}{T}
        \lambda (\{ t \leq T; \, h_{\delta}(t) \geq \alpha \}) = 0.
        $$

\begin{proof}
Fix arbitrary $\gamma > 1$ and $K > 1$ and let
        $$
        B_{\gamma}^K = \left\{ t; \, \frac{\mu(\gamma I_{t})}{\mu(I_{t})} \geq K \right\}.
        $$
Choose $L = L(\gamma) > 0$ and $\delta_0 = \delta_0(\gamma) > 0$ so small that the inequalities
                $$
                e^{2\delta} \leq 1 + 3\delta, \quad
                3\delta \leq e^{-L}, \quad
                e^L(1 + 3\delta) \leq \gamma
                $$
are satisfied for all $\delta \in (0, \, \delta_0)$.

Take any $\delta \in (0, \delta_0)$. By Chebyshev's inequality
        $$
        \lambda(\{ t \leq T; \, h_{\delta}(t) \geq \alpha \}) \leq
        \lambda(B_{\gamma}^K \cap [0, \, T]) + 
        \frac{1}{\alpha} \int_{(B_{\gamma}^K)^c \cap [0, \, T]} h_{\delta} \intd{\lambda},
        $$
and the integral can be estimated by
        \begin{align*}
        \int_{(B_{\gamma}^K)^c \cap [0, \, T]} h_{\delta} \intd{\lambda}
        &=
        \int_{(B_{\gamma}^K)^c \cap [\delta, \, T + \delta]}
        \frac{\mu(A_{t - 2\delta})}{\mu(I_t)} \intd{\lambda(t)} -
        \int_{(B_{\gamma}^K)^c \cap [-\delta, \, T - \delta]}
        \frac{\mu(A_{t + 2\delta})}{\mu(I_t)} \intd{\lambda(t)} \\
        &\leq
        (K + 1)\delta +
        \int_{(B_{\gamma}^K)^c \cap [0, \, T]} \frac{\mu(A_{t - 2\delta}) -
        \mu(A_{t + 2\delta})}{\mu(I_t)} \intd{\lambda(t)},
        \end{align*}
using that for every $t \in (B_{\gamma}^K)^c$
        $$
        \frac{\mu(A_{t - 2\delta})}{\mu(I_t)} \leq K
        \quad \text{ and } \quad
        \frac{\mu(A_{t + 2\delta})}{\mu(I_t)} \leq 1
        $$
(the first inequality holds because $e^{2\delta} \leq \gamma$ so that
$A_{t - 2\delta} \subset e^{2\delta}I_t \subset \gamma I_t$).
Since $0 \in A$, the numerator in the last integral can be written as
        $$
        \mu(A_{t - 2\delta}) - \mu(A_{t + 2\delta}) =
        \mu(A_{t - 2\delta} \setminus A_{t + 2\delta}) =
        \mu(J_a(t)) + \mu(J_b(t)),
        $$
where $a$ and $b$ are the endpoints of $A$ and 
$J_c(t) = [ce^{-(t \pm 2\delta)}]$ for $c \in \{ a, \, b\}$. For each $c$,
        \begin{align*}
        &\int_{(B_{\gamma}^K)^c \cap [0, \, T]} \frac{\mu(J_c(t))}{\mu(I_t)} \intd{\lambda(t)} =
        \int_{(B_{\gamma}^K)^c \cap [0, \, T]}
        \frac{\mu(\gamma I_t)}{\mu(I_t)} \frac{\mu(J_c(t))}{\mu(\gamma I_t)} \intd{\lambda(t)} \leq \\
        K &\int_{(B_{\gamma}^K)^c \cap [0, \, T]} \frac{\mu(J_c(t))}{\mu(\gamma I_t)} \intd{\lambda(t)} \leq
        K \sum_{k = 0}^{\floor{T / L}} \frac{1}{\mu(\gamma I_{(k + 1)L})}
        \int_{kL}^{(k + 1)L} \mu(J_c(t)) \intd{t}.
        \end{align*}
Suppose that
$c > 0$ so that $J_c(t) = [ce^{-(t + 2\delta)}, \, ce^{-(t - 2\delta)}]$
(if $c < 0$ then the endpoints come in the opposite order, but the computations
are similar). Then for each $k$,
        \begin{align*}
        \int_{kL}^{(k + 1)L} \mu(J_c(t)) \intd{t} &= \left[ s = ce^{-t} \right] =
        \int_{ce^{-(k + 1)L}}^{ce^{-kL}} \frac{1}{s} \mu([se^{-2\delta}, \, se^{2\delta}]) \intd{s} \\
        &\leq
        c^{-1}e^{(k + 1)L} \int_{ce^{-(k + 1)L}}^{ce^{-kL}}
        \mu([s - 3\delta ce^{-kL} , \, s + 3\delta ce^{-kL}]) \intd{s} \\
        &\leq
        6\delta e^L
        \mu([ce^{-kL}(e^{-L} - 3\delta), \, ce^{-kL}(1 + 3\delta)]),
        \end{align*}
where the last inequality is by Lemma \ref{tubelemma}. By the choice of $L$ and $\delta_0$, the interval
that appears in the last member is included in $\gamma I_{(k + 1)L}$, so
        $$
        \frac{1}{\mu (\gamma I_{(k + 1)L})} \int_{kL}^{(k + 1)L} \mu(J_c(t)) \intd{t}
        \leq 6 \delta e^L.
        $$
Thus
        \begin{align*}
        &\limsup_{T \to \infty} \frac{1}{T} \lambda (\{ t \leq T; \, h_{\delta}(t) \geq \alpha \}) \leq \\
        &\limsup_{T \to \infty} \frac{1}{T} \left(
        \lambda(B_{\gamma}^K\cap [0, \, T]) + \frac{1}{\alpha}
        \left(
        (K + 1) \delta + \frac{T}{L} 12 K \delta e^L
        \right) \right)
        = \\
        &\limsup_{T \to \infty} \frac{1}{T} \lambda(B_{\gamma}^K \cap [0, \, T])
        + \frac{12Ke^L}{\alpha L} \delta.
        \end{align*}
Letting $\delta \to 0$ gives
        $$
        \limsup_{\delta \to 0} \limsup_{T \to \infty} \frac{1}{T} \lambda (\{ t \leq T; \, h_{\delta}(t) \geq \alpha \}) \leq
        \limsup_{T \to \infty} \frac{1}{T} \lambda(B_{\gamma}^K \cap [0, \, T]),
        $$
and since this holds for any $\gamma > 1$ and $K > 1$ the lemma follows from \eqref{TC}
when $\gamma \to 1^+$ and $K \to \infty$.
\end{proof}
\end{lemma}

\begin{lemma}   \label{difflemma}
Suppose that $\mu \in \mathcal{M}(\rea)$ satisfies the condition \eqref{TC} at some $x \in \supp \mu$.
Then for any $\varphi \in C(I)$ and any $\alpha > 0$,
        \begin{equation}        \label{badtimeseq}
        \lim_{T \to \infty} \frac{1}{T} \lambda(B(\varphi, \alpha) \cap [0, \, T]) = 0,
        \end{equation}
where
        $$
        B(\varphi, \alpha) = 
        \left\{ t; \, \left|\scfl{\mu}{x}{t}(\varphi) - \scfl{(f\mu)}{f(x)}{t - s}(\varphi)\right|
        \geq \alpha \right\}
        $$
and $s = \log f'(x)$.

\begin{proof}
After a translation of the real line it may be assumed that $x = 0$.
The diffeomorphism $f$ can be written on the form $f(x) = f(0) + f'(0)f_0(x)$
where $f_0$ is a diffeomorphism with $f_0(0) = 0$ and $f_0'(0) = 1$. Since
        $
        \scfl{(f\mu)}{f(0)}{t - s} = \scfl{(f_0 \mu)}{0}{t},
        $
it is enough to prove the lemma with $f_0$ in place of $f$. It will first
be shown that \eqref{badtimeseq} holds for $\varphi = \chi_A$,
where $A \subset I$ is an interval containing $0$.

Now, $f_0^{-1}$ has the form $f_0^{-1}(x) = x(1 + r(x))$ where $\lim_{x \to 0} r(x) = 0$.
Let $R(t) = \sup_{\tau \geq t} |r(e^{-\tau})|$. Then from the inclusions
        \begin{align*}
        (1 - R(t))A_t \subset A_t&, \, f^{-1}(A_t) \subset (1 + R(t))A_t \\
        (1 - R(t))I_t \, \subset \, I_t&, \, f^{-1}(I_t) \, \subset (1 + R(t))I_t
        \end{align*}
it follows that
        $$
        \frac{\mu((1 - R(t))A_t)}{\mu((1 + R(t))I_t)}
        \leq
        \scfl{\mu}{0}{t}(A), \,
        \scfl{(f\mu)}{0}{t}(A)
        \leq
        \frac{\mu((1 + R(t))A_t)}{\mu((1 - R(t))I_t)},
        $$
and thus
        \begin{align}
        \left| \scfl{\mu}{0}{t}(A) - \scfl{(f\mu)}{0}{t}(A) \right|
        \leq
        \frac{\mu((1 + R(t))A_t)}{\mu((1 - R(t))I_t)}
        -
        \frac{\mu((1 - R(t))A_t)}{\mu((1 + R(t))I_t)}.
        \label{Rrhs}
        \end{align}
Let $\varepsilon > 0$ and define $h_{\delta}$ as in Lemma \ref{hdeltalemma}. Then there
is a $\delta_1 > 0$ such that
        $$
        \limsup_{T \to \infty} \frac{1}{T} \lambda (\{ t \leq T; \, h_{\delta_1}(t) \geq \alpha \}) < \varepsilon.
        $$
Since $\lim_{t \to \infty} R(t) = 0$, there is then a $T_1$ such that the right member of \eqref{Rrhs}
is less than $h_{\delta_1}(t)$ for $t > T_1$. Thus
        \begin{align*}
        &\limsup_{T \to \infty} \frac{1}{T} \lambda(\{ B(\chi_A, \alpha) \cap [0, \, T] \})
        \leq \\
        &\limsup_{T \to \infty} \frac{1}{T} \left( \lambda([0, \, T_1])
        + \lambda (\{ t \leq T; \, h_{\delta_1}(t) \geq \alpha \}) \right) < \varepsilon.
        \end{align*}
This holds for any $\varepsilon > 0$, so \eqref{badtimeseq} is proven for $\varphi = \chi_A$.

It is not too hard to see that the set of $\varphi$ that satisfy \eqref{badtimeseq} for all $\alpha > 0$
is a linear space. If $A \subset I$ is an interval that does not contain $0$, then $A$ is the difference
of two intervals that do, so \eqref{badtimeseq} holds for $\varphi = \chi_A$ for
\emph{any} interval $A \subset I$. Thus \eqref{badtimeseq} holds whenever $\varphi$ is a step function,
and the lemma follows since any $\varphi \in C(I)$ can be uniformly approximated by step functions.
\end{proof}
\end{lemma}

The uniform structure of $\mathcal{M}^{\square}$ will be used in the proof of Lemma
\ref{TClemma} --- therefore some facts about uniform spaces are recalled first.
A \emph{uniform space} is a set $X$ with a notion of uniform closeness.
This is represented by \emph{entourages}, which are subsets of $X \times X$ satisfying
some appropriate axioms. The idea is that an entourage $V$ should contain all pairs
of points that are ``$V$-close'' to each other. A \emph{fundamental system of entourages}
is a set $\Phi$ of entourages such that any entourage includes an entourage from $\Phi$.
As an example, if $X$ has a metric $d$ then a fundamental system of entourages
for the metric uniformity on $X$ is given by the sets
	$$
	\{ (x_1, x_2); \, d(x_1, x_2) < \alpha \}, \qquad \alpha > 0;
	$$
such an entourage represents all pairs of points that are $\alpha$-close.
The uniform structure induces a topology on $X$, namely the one where the neighbourhood
filter at $x$ consists of the sets $\{ y; \, (x, y) \in V \}$ as $V$ runs through all
entourages.

If $X$ and $Y$ are uniform spaces, then a function $f: X \to Y$ is
\emph{uniformly continuous} if for any entourage $W$ of $Y$ there is an entourage $V$ of $X$
such that $(f(x_1), f(x_2)) \in W$ whenever $(x_1, x_2) \in V$. It can be shown that
any uniformly continuous function is continuous, and that any continuous function from
a compact uniform space to a uniform space is uniformly continuous.

The weak-$*$ topology on $\mathcal{M}^{\square}$ is the initial topology with respect
to the maps
        $$
        \nu \mapsto
        \int \varphi \intd{\nu}
        , \qquad \varphi \in C(I).
        $$
This topology is also induced by the uniformity determined by the fundamental system of
entourages consisting of all finite intersections of sets of the form
        $$
        \left\{ (\nu_1, \nu_2); \, \left|
        \int \varphi \intd{\nu_1} - \int \varphi \intd{\nu_2}
        \right| < \alpha \right\},
        $$
where $\alpha > 0$ and $\varphi \in C(I)$.

\begin{proof}[Proof of Lemma \ref{TClemma}]
Take any $g \in C(\mathcal{M}^{\square})$ --- then $g$ is bounded and uniformly
continuous since $\mathcal{M}^{\square}$ is compact. To prove the lemma, it
is enough to show that
        $$
        \lim_{T \to \infty} \frac{1}{T} \int_{0}^T
        \left|g\left(\scfl{\mu}{x}{t}\right) - g\left(\scfl{(f\mu)}{f(x)}{t - s}\right) \right|
        \intd{t} = 0,
        $$
where $s = \log f'(x)$.

Let $\varepsilon > 0$. Then there is an entourage $V$ of $\mathcal{M}^{\square}$ such
that $|g(\nu_1) - g(\nu_2)| < \varepsilon$ whenever $(\nu_1, \nu_2)$ is in $V$. Now, $V$
includes a set of the form $\cap_{i = 1}^{n} V_i$, where
        $$
        V_i = \left\{ (\nu_1, \nu_2); \, \left|
        \nu_1(\varphi_i) - \nu_2(\varphi_i)
        \right| < \alpha \right\},
        $$
for some $\alpha > 0$ and some $\varphi_1, \ldots, \varphi_n \in C(I)$,
since the sets of this form constitute a fundamental system of entourages of the
uniformity on $\mathcal{M}^{\square}$. Let
        $$
        B = \cup_{i = 1}^{n} \{ t; \, \scfl{\mu}{x}{t} \in V_i^c \}.
        $$
Then
        \begin{align*}
        &\frac{1}{T} \int_{0}^{T}
        \left|g\left(\scfl{\mu}{x}{t}\right) - g\left(\scfl{(f\mu)}{f(x)}{t - s}\right)\right|
        \intd{t}
        \leq \\
        &\frac{\|g\|_{\infty}}{T} \lambda(B\cap[0, \, T]) +
        \frac{1}{T} \int_{B^c \cap [0, \, T]}
        \left| g\left(\scfl{\mu}{x}{t}\right) - g\left(\scfl{(f\mu)}{f(x)}{t - s}\right) \right| 
        \intd{\lambda(t)}
        \leq \\
        &\frac{\|g\|_{\infty}}{T} \lambda(B\cap[0, \, T]) + \varepsilon.
        \end{align*}
When $T \to \infty$, the last expression goes to
$\varepsilon$ by Lemma \ref{difflemma}, and since $\varepsilon$ is arbitrary
the lemma follows.
\end{proof}

\section{Motivating Examples}	\label{examplessec}
This section presents two examples of what can go wrong in Lemma 2.3 to
motivate the altered formulation given in Lemma \ref{TClemma}, where the
conclusion is changed from ``\ldots then $\scfl{\mu}{x}{t}$ and $\scfl{(f\mu)}{f(x)}{t - s}$
are asymptotic'' to ``\ldots then $\scdi{\mu}{x}{T}$ and $\scdi{f\mu}{f(x)}{T}$ are asymptotic''
and an extra condition is added to the hypothesis.

\begin{example}	\label{example1}
Let
	$$
	\mu = \sum_{k = 1}^{\infty} 2^{-k} \delta_{2^{-k}}
	$$
and define $f \in \diffk{1}$ by
	$$
	f^{-1}(x) = x - x^3;
	$$
then $f(0) = 0$ and $f'(0) = 1$.
For all $n$,
	$$
	\mu([-2^{-n}, \, 2^{-n}]) =
	\sum_{k = n}^{\infty} 2^{-k} = 2^{-(n - 1)}
	$$
and
	$$
	\mu([f^{-1}(-2^{-n}), \, f^{-1}(2^{-n})]) =
	\sum_{k = n + 1}^{\infty} 2^{-k} = 2^{-n},
	$$
since $2^{-(n + 1)} < f^{-1}(2^{-n}) < 2^{-n}$.
Also for all $n$ and all $c$ in some small neighbourhood of $3 / 4$,
	$$
	2^{-(n + 1)} < f^{-1}(c2^{-n}) < c2^{-n} < 2^{-n},
	$$
and thus
	$$
	\mu([0, \, f^{-1}(c2^{-n})]) = \sum_{k = n + 1}^{\infty} 2^{-k} =
	\mu([0, \, c2^{-n}]).
	$$
By considering $\varphi \in C(I)$ such that
$\chi_{[-1, \, 3 / 4]} \leq \varphi \leq \chi_{[-1, \, 3 / 4 + \varepsilon]}$
for some small $\varepsilon$, one sees that $\scfl{\mu}{0}{t}$ and $\scfl{(f\mu)}{0}{t}$ are
not asymptotic, since for $t_n = n\log2$,
	\begin{align*}
	\scfl{(f\mu)}{0}{t_n}(\varphi) - \scfl{\mu}{0}{t_n}(\varphi) &\geq
	\scfl{(f\mu)}{0}{t_n}([0, \, 3 / 4]) - \scfl{\mu}{0}{t_n}([0, \, 3 / 4 + \varepsilon]) \\
	&=
	\frac{2^{-n}}{2^{-n}} - \frac{2^{-n}}{2^{-(n - 1)}}= \frac{1}{2}.
	\end{align*}
\end{example}

The example still works if the weight of the $k$:th point mass is changed to
$w^{-k}$ for any $w > 1$. The local dimension of $\mu$ at $0$ then becomes
$\log w / \log 2$, which can take any value in $(0, \, \infty)$.
It would also be possible to replace the point masses by measures concentrated
on the intervals $(f^{-1}(2^{-k}), \, 2^{-k})$. Thus one could give $\mu$ any
prescribed (``global'') dimension, for example by putting a suitable Cantor measure
on each interval. In particular, it is possible to make $\mu$ non-atomic --- in
this case the map $t \mapsto \scfl{\mu}{0}{t}$ is continuous but changes rapidly near
each $t_n$.

In Example \ref{example1}, the measures $\scfl{\mu}{x}{t}$ and
$\scfl{(f\mu)}{f(x)}{t - s}$ are far apart only for very short time periods,
and it is still true that $\scdi{\mu}{x}{T}$ and $\scdi{f\mu}{f(x)}{T}$ are asymptotic.
The next example shows that even this can fail if $\scfl{\mu}{x}{t}$ has a lot of
mass near $\{ \pm1 \}$ for \emph{all} $t$. It is a special case of Proposition \ref{failproposition}.

\begin{example}	\label{example3}
Let $\mu$ be the measure on \rea given by the density function $g(x) = \frac{1}{x^2}e^{-\frac{1}{|x|}}$,
so that
	$$
	\mu(\pm[0, x]) = e^{-\frac{1}{x}} =: G(x) \quad \text{for } x > 0
	$$
(and $\mu(\{ 0 \}) = 0$). If $c \in (0, \, 1)$ then
	$$
	\lim_{x \to 0^+} \frac{G(cx)}{G(x)} =
	\lim_{x \to 0^+} \exp \left( -\frac{1}{x}\left( \frac{1}{c} - 1 \right) \right) = 0,
	$$
so for $c \in (0, \, 1]$
	$$
	\scfl{\mu}{0}{t}(\pm[0, \, c]) = \frac{G(ce^{-t})}{2G(e^{-t})} \to
	\begin{cases}
	0			& \text{ if } c < 1 \\
	\frac{1}{2}	& \text{ if } c = 1
	\end{cases},
	\quad
	t \to \infty.
	$$
Thus
	$$
	\lim_{t \to \infty} \scfl{\mu}{0}{t} = \frac{1}{2}\left( \delta_{-1} + \delta_{1} \right)
	.
	$$

Define $f \in \diffk{1}$ by letting
	$$
	f^{-1}(x) =
	\begin{cases}
	x 		& \text{ if } x \leq 0 \\
	x + x^2 &	\text{ if } x > 0;
	\end{cases}
	$$
then $f(0) = 0$ and $f'(0) = 1$. For $c \in (0, \, 1)$,
	$$
	\scfl{(f\mu)}{0}{t}(\pm[0, \, c]) 
	\leq
	\frac{G(ce^{-t} + c^2e^{-2t})}{2G(e^{-t})} \leq
	\frac{G(\sqrt{c}e^{-t})}{G(e^{-t})} \to 0, \quad t \to \infty,
	$$
where the second inequality holds for large enough $t$. From
	$$
	\frac{G(x + x^2)}{G(x)} =
	\exp \left(\frac{1}{x} - \frac{1}{x + x^2} \right) =
	\exp \left( \frac{1}{1 + x} \right) \to e,
	\quad x \to 0
	$$
it follows that
	$$
	\scfl{(f\mu)}{0}{t}([-1, \, 0]) = \frac{G(e^{-t})}{G(e^{-t}) + G(e^{-t} + e^{-2t})}
	\to \frac{1}{1 + e}, \quad t \to \infty.
	$$
Thus
	$$
	\lim_{t \to \infty} \scfl{(f\mu)}{0}{t} = \frac{1}{1 + e} \delta_{-1} + \frac{e}{1 + e} \delta_1
	.
	$$

Since $\scfl{\mu}{0}{t}$ and $\scfl{(f\mu)}{0}{t}$ converge to different measures they are not asymptotic,
and in fact they generate two different point mass distributions.
\end{example}

\section{The Implications of Lemma \ref{TClemma}}	\label{applicationsection}
Here we will prove Propositions \ref{fintedimproposition}, \ref{succeedproposition} and \ref{failproposition}
and then discuss the gap between Propositions~\ref{succeedproposition} and \ref{failproposition}.

\begin{proof}[Proof of Proposition~\ref{fintedimproposition}]
It follows from $\udim{\mu}{x} < \infty$ that $x \in \supp \mu$. For $\gamma > 1$ let
	$$
	B_{\gamma} = \left\{ t; \, \frac{\mu(x + \gamma I_t)}{\mu(x + I_t)} \geq 2 \right\}.
	$$

If $h: [0, \infty) \to [0, \infty)$ is any decreasing function and $\varepsilon$, $T$ and $\alpha$ are
positive numbers, then
	$$
	\lambda(\{ t \in [\varepsilon, \, T]; \, f(t - \varepsilon) - f(t) \geq \alpha \}) \leq
	\frac{\varepsilon}{\alpha}(f(0) - f(T)).
	$$
Applying this to $h(t) = \log(\mu(x + I_t))$ and $\varepsilon = \log \gamma$ gives 
(since $\gamma I_{t} = I_{t - \log \gamma}$)
	\begin{align*}
	\limsup_{T \to \infty} \frac{1}{T} \lambda (B_{\gamma} \cap [0, \, T]) &\leq
	\limsup_{T \to \infty} \frac{\log \gamma}{T}
	\left(
	1 + \frac{\log(\mu(x + I_0)) - \log(\mu(x + I_{T}))}{\log 2}
	\right) \\
	&=
	\limsup_{T \to \infty} \frac{\log \gamma}{\log 2} \frac{-\log(\mu(x + I_T))}{T}
	= \frac{\log \gamma}{\log 2} \udim{\mu}{x}.
	\end{align*}
Letting $\gamma \to 1^+$ shows that the condition \eqref{TC} is satisfied, so the proposition follows
by Lemma \ref{TClemma}.
\end{proof}

The following general fact about the weak-$*$ topology will be used to prove Proposition
\ref{succeedproposition} and Proposition \ref{failproposition}.

\begin{lemma}	\label{opensetlemma}
Let $X$ be a locally compact Hausdorff space that has a countable base for the topology,
and let $\mathcal{P}(X)$ be the space of Borel probability measures on $X$ with the
weak-$*$ topology. Let $V \subset X$ be open and let $c \in \rea$. Then the set
	$$
	\{ \nu \in \mathcal{P}(X); \, \nu(V) > c \}
	$$
is open.

\begin{proof}
Suppose that $\nu(V) > c$. Any finite Borel measure on a locally compact space with a
countable base is inner regular
,
so there is a compact set $K \subset V$ such that $\nu(K) > c$. By Urysohn's lemma,
there is then a continuous function $\varphi: X \to \rea$ such that
$\chi_K \leq \varphi \leq \chi_V$. Thus $\{ \nu; \, \nu(\varphi) > c \}$ is
an open set that contains $\nu$ and is included in $\{ \nu; \, \nu(V) > c \}$,
which proves the lemma.
\end{proof}
\end{lemma}

\begin{proof}[Proof of Proposition~\ref{succeedproposition}]
It will be shown that the condition \eqref{TC} is satisfied, so that the proposition
follows from Lemma \ref{TClemma}. After a translation of the real line it may
be assumed that $x = 0$.

Let $\Omega$ be the $\omega$-limit set of $\scdi{\mu}{0}{T}$. Then for every open set $V$
that includes $\Omega$, there is a $T_V$ such that $\scdi{\mu}{0}{T} \in V$ for all
$T \geq T_V$. For if this was not the case, then there would be a sequence $T_k \to \infty$
such that $\scdi{\mu}{0}{T_k} \in V^c$ for all $k$. Since the space of distributions
is compact there would be a converging subsequence, and the limit would lie in
$\Omega \cap V^c$, contradicting the assumption that $\Omega \subset V$.

For any $\gamma > 1$ and $K > 1$, let
	$$
	C_{\gamma}^K = \left\{
	\nu \in \mathcal{M}^{\square}; \,
	\nu\left( \left( -\frac{1}{\gamma}, \, \frac{1}{\gamma} \right)\right) \leq \frac{1}{K}
	\right\};
	$$
this is a closed set by Lemma \ref{opensetlemma}. Let
$\varepsilon_{\gamma}^K = \sup_{P \in \Omega} P(C_{\gamma}^K)$.
By Lemma \ref{opensetlemma} again, the set
	$$
	\left\{ P; \, P(C_{\gamma}^K) < \varepsilon_{\gamma}^K + \frac{1}{K} \right\} =
	\left\{ P; \, P((C_{\gamma}^K)^c) > 1 - \left( \varepsilon_{\gamma}^K + \frac{1}{K} \right) \right\}
	$$
is an open set of distributions. Since it includes $\Omega$ (the reason for having the term $\frac{1}{K}$ is to make
this true in case $\varepsilon_{\gamma}^K = 0$), it contains $\scdi{\mu}{0}{T}$ for all large enough $T$.
Hence,
	\begin{align*}
	\varepsilon_{\gamma}^K + \frac{1}{K} \geq \limsup_{T \to \infty} \scdi{\mu}{0}{T}(C_{\gamma}^K)
	&\stackrel{*}{=}
	\limsup_{T \to \infty} \frac{1}{T} \lambda\left(
	\left\{ t \leq T; \, \scfl{\mu}{0}{t}\left(\frac{1}{\gamma} I \right) \leq \frac{1}{K} \right\}
	\right) \\
	&=
	\limsup_{T \to \infty} \frac{1}{T} \lambda\left(
	\left\{ t \leq T; \, \frac{\mu(\frac{1}{\gamma} I_t )}{\mu(I_t)} \leq \frac{1}{K} \right\}
	\right) \\
	&=
	\limsup_{T \to \infty} \frac{1}{T} \lambda\left(
	\left\{ t \leq T; \, \frac{\mu(I_t )}{\mu(\gamma I_t)} \leq \frac{1}{K} \right\}
	\right) \\
	&=
	\limsup_{T \to \infty} \frac{1}{T} \lambda\left(
	\left\{ t \leq T; \, \frac{\mu(\gamma I_t )}{\mu(I_t)} \geq K \right\}
	\right)
	\end{align*}
(at $*$, the open interval $\frac{1}{\gamma} (-1, \, 1)$ is replaced by the closed interval
$\frac{1}{\gamma} I$. This does not make any difference since the set of $t$ for which
$\scfl{\mu}{0}{t}$ gives positive measure to $\{ \pm \frac{1}{\gamma} \}$ has Lebesgue measure $0$).
Thus to prove the proposition, it suffices to show that
$\lim_{K \to \infty} \limsup_{\gamma \to 1^+} \varepsilon_{\gamma}^K = 0$.

So take any $\varepsilon > 0$. Since $C_{\gamma}^K$ decreases if $\gamma$ decreases or
$K$ increases and $P(\bigcap_{K, \gamma} C_{\gamma}^K) = 0$ for any $P \in \Omega$,
the sets $\{\{ P; \, P(C_{\gamma}^K) < \varepsilon \}\}_{\gamma, K > 1}$
cover $\Omega$. Since they are open and $\Omega$ is compact there is a finite subcover,
say
	$$
	\Omega \subset \bigcup_{k = 1}^n \{ P; \, P(C_{\gamma_k}^{K_k}) < \varepsilon \}.
	$$
Now, if $\gamma$ decreases or $K$ increases then
$\{ P; \, P(C_{\gamma}^K) < \varepsilon \}$ \mbox{\emph{in}creases}. Thus $\Omega$
is in fact covered by $\{ P; \, P(C_{\gamma}^K) < \varepsilon \}$ whenever
$\gamma \leq \min_k \gamma_k$ and $K \geq \max_k K_k$. Hence
	$$
	\limsup_{K \to \infty} \limsup_{\gamma \to 1^+} \sup_{P \in \Omega} P(C_{\gamma}^K) \leq \varepsilon,
	$$
and since $\varepsilon$ is arbitrary this concludes the proof.
\end{proof}

\begin{proof}[Proof of Proposition~\ref{failproposition}]
After a translation of the real line it may be assumed that $x = 0$. Let $\beta_0 \in (0, \, 1/2)$
be such that
$P(\{ (1 - \beta)\delta_{-1} + \beta \delta_{1}; \, \beta_0 < \beta < 1 - \beta_0 \}) =: d > 0$
and set $c = (2 - \beta_0) / (2 - \beta_0^2)$ (thus $c \in (0, \, 1)$). For $n = 1, 2, \ldots$ let $B_n$ be
the subset of $\mathcal{M}^{\square}$ defined by
	\begin{align*}
	B_n = \{ \nu \in \mathcal{M}^{\square}; \,
	&\nu([-1, \, -1 + 2^{-n}) \cup (1 - 2^{-n}, \, 1]) > c, \\
	&\nu([-1, \, -1 + 2^{-n})) > \beta_0, \\
	&\nu((1 - 2^{-n}, \, 1]) > \beta_0
	\}.
	\end{align*}
Let $T'_1, T'_2, \ldots$ be a sequence increasing to infinity such that $\scdi{\mu}{x}{T'_k} \to P$.
Each $B_n$ includes $\{ (1 - \beta)\delta_{-1} + \beta \delta_{1}; \, \beta_0 < \beta < 1 - \beta_0 \}$
and is open by Lemma \ref{opensetlemma}, so
	$$
	\liminf_{k \to \infty} \scdi{\mu}{x}{T'_k}(B_n) \geq P(B_n) \geq d,
	$$
and thus there is a subsequence $T_1, T_2, \ldots$ such that $\scdi{\mu}{x}{T_n}(B_n) \geq d / 2$
for all $n$. By passing to a subsequence again if necessary, it can be assumed that
$e^{-T_n}(1 - 2^{-n})$ is a decreasing sequence of numbers.

Now put $f'(x) = 1$ if $x \leq 0$ and define $f'(x)$ for $x > 0$ by setting
	$$
	f'\left( \frac{e^{-T_n}(1 - 2^{-n})}{2} \right) = \frac{2^n + 1}{2^n - 1}
	$$
and interpolating linearly. Then $f'$ is increasing, continuous and $\geq 1$. In
particular, $f(x) = \int_0^{x} f'(s) \intd{s}$ is a diffeomorphism with $f(x) = x$
for $x \leq 0$. If $x \geq e^{-T_n}$ then
	\begin{align*}
	f(x(1 - 2^{-n})) &\geq \frac{x(1 - 2^{-n})}{2} \left( 1 + f'\left( \frac{x(1 - 2^{-n})}{2} \right) \right)\\ & \geq
	\frac{x(1 - 2^{-n})}{2}
	\left( 1 + \frac{2^n + 1}{2^n - 1} \right)
	= x.
	\end{align*}
If follows that $f^{-1}(e^{-t}) \leq e^{-t}(1 - 2^{-n})$ for all $t \leq T_n$.

Suppose that $t \leq T_n$ is such that $\scfl{\mu}{0}{t} \in B_n$ and let
	$$
	\beta = \frac{\scfl{\mu}{0}{t}((1 - 2^{-n}, \, 1])}
	{\scfl{\mu}{0}{t}([-1, \, -1 + 2^{-n}) \cup (1 - 2^{-n}, \, 1])};
	$$
then $\beta \in (\beta_0, \, 1 - \beta_0)$ and
	\begin{align*}
	c(1 - \beta) < &\scfl{\mu}{0}{t}([-1, \, -1 + 2^{-n})) < 1 - \beta, \\
	c\beta < &\scfl{\mu}{0}{t}((1 - 2^{-n}, \, 1]).
	\end{align*}
It follows that
	\begin{equation}	\label{longcalceq}
	\begin{split}
	\scfl{(f\mu)}{0}{t}&([-1, \, -1 + 2^{-n}))
	=
	\frac{\mu(f^{-1}(e^{-t}[-1, \, -1 + 2^{-n})))}{\mu(f^{-1}(I_t))} \\
	&=
	\frac{\mu(e^{-t}[-1, \, -1 + 2^{-n}))}{\mu([-e^{-t}, \, f^{-1}(e^{-t})])}
	\geq
	\frac{\mu(e^{-t}[-1, \, -1 + 2^{-n}))}{\mu(e^{-t}[-1, \, 1 - 2^{-n}])} \\
	&\geq
	\frac{c(1 - \beta)}{1 - c\beta}
	\geq
	\frac{c(1 - \beta_0)}{1 - c\beta_0}
	=
	1 - \frac{\beta_0}{2}
	\geq
	\scfl{\mu}{0}{t}([-1, \, -1 + 2^{-n})) + \frac{\beta_0}{2}.
	\end{split}
	\end{equation}

Let $\varphi$ be a continuous function on $I$ such that
$\chi_{[-1, \, -1/2]} \leq \varphi \leq \chi_{[-1, \, 0]}$ and define
$g \in C(\mathcal{M}^{\square})$ by $g(\nu) = \nu(\varphi)$. Then
$g(\scfl{(f\mu)}{0}{t}) \geq g(\scfl{\mu}{0}{t})$ for all $t$, and if $t \leq T_n$
and $\scfl{\mu}{0}{t} \in B_n$ then $g(\scfl{(f\mu)}{0}{t}) \geq g(\scfl{\mu}{0}{t}) + \beta_0 / 2$
by \eqref{longcalceq}. Thus
	\begin{align*}
	\scdi{f\mu}{f(0)}{T_n}(g) - \scdi{\mu}{0}{T_n}(g)
	&=
	\frac{1}{T_n} \int_{0}^{T_n} g(\scfl{(f\mu)}{0}{t}) - g(\scfl{\mu}{0}{t}) \intd{t} \\
	&\geq
	\frac{1}{T_n} \int_{\{ t \leq T_n; \, \scfl{\mu}{0}{t} \in B_n \}}
	g(\scfl{(f\mu)}{0}{t}) - g(\scfl{\mu}{0}{t}) \intd{t} \\
	&\geq
	\frac{\beta_0}{2T_n} \lambda(\{ t \leq T_n; \, \scfl{\mu}{0}{t} \in B_n \})
	\geq
	\frac{\beta_0 d}{4}
	\end{align*}
for all $n$. This shows that $\scdi{\mu}{0}{T}$ and $\scdi{f\mu}{0}{T}$ are not asymptotic.
\end{proof}

There is a discrepancy between the condition in Proposition \ref{failproposition}
allowing $\scdi{\mu}{x}{T}$ and $\scdi{f\mu}{f(x)}{T}$ not to be asymptotic,
and the condition in Proposition \ref{succeedproposition} ensuring that they are. Namely,
if $P(\{ (1 - \beta)\delta_{-1} + \beta \delta_1; \, 0 < \beta < 1 \}) = 0$ for all
$\omega$-limit points $P$ of $\scdi{\mu}{x}{T}$ but there is some $\omega$-limit point
$P_0$ such that $P_0(\{ \delta_{-1}, \, \delta_{1} \}) > 0$, then neither
Proposition \ref{failproposition} nor Proposition \ref{succeedproposition} applies. In this
case, it can happen either that $\scdi{\mu}{x}{T}$ and $\scdi{f\mu}{f(x)}{T}$ are
asymptotic for all $f \in \diffk{1}$ or that there is some $f$ for which
they are not asymptotic, depending on $\mu$ and $x$. This is illustrated
in the next two examples.

\begin{example}	\label{example4}
For $x > 0$ let
	$$
	G(x) = e^{-\frac{1}{x}}
	$$
and let $\mu$ be the measure on \rea given by $\mu((-\infty, \, 0]) = 0$ and
	$$
	\mu([0, \, x]) = G(x), \qquad x > 0.
	$$
As in example \ref{example3},
	$$
	\lim_{x \to 0^+} \frac{G(cx)}{G(x)} = 0
	$$
for any $c \in (0, \, 1)$.

Take any $f \in \diffk{1}$ such that $f(0) = 0$ and $f'(0) = 1$ --- then $f^{-1}$ has
the form $f^{-1}(x) = x(1 + r(x))$ where $\lim_{x \to 0} r(x) = 0$. Thus for any
$c \in (0, \, 1)$
	\begin{align*}
	\scfl{(f\mu)}{0}{t} ([-1, \, c]) &=
	\frac{\mu([f^{-1}(-e^{-t}), \, f^{-1}(ce^{-t})])}{\mu([f^{-1}(-e^{-t}), \, f^{-1}(e^{-t})])}
	=
	\frac{G(ce^{-t}(1 + r(ce^{-t})))}{G(e^{-1}(1 + r(e^{-t})))} \\
	&\leq
	\frac{G(c^{2 / 3}e^{-t})}{G(c^{1/3}e^{-t})} \to 0, \quad t \to \infty,
	\end{align*}
where the inequality holds for all large enough $t$. Thus
$\lim_{t \to \infty}\scfl{(f\mu)}{0}{t} = \delta_1$.

Now take \emph{any} $f \in \diffk{1}$. Then $f$ can be written on the form
$f(x) = f(0) + f'(0)f_0(x)$ where $f_0$ is a diffeomorphism with $f_0(0) = 0$ and
$f_0'(0) = 1$. Since $\scfl{(f\mu)}{f(0)}{t - s} = (f_0 \mu)_{0, t}$ where $s = \log f'(0)$
it follows from the preceding paragraph that $\lim_{t \to \infty}\scfl{(f\mu)}{0}{t} = \delta_1$.
This holds in particular if $f$ is the identity map and thus $\scfl{\mu}{0}{t}$ and
$\scfl{(f\mu)}{f(0)}{t - s}$ are asymptotic for any $f \in \diffk{1}$.
\end{example}

\begin{example}	\label{example5}
For $x > 0$ let
	$$
	G(x) = \exp\left(-\exp\left(\frac{1}{x - x^2}\right)\right), \qquad
	H(x) = \exp\left(-\exp\left(\frac{1}{x + x^2}\right)\right).
	$$
Since both $G$ and $H$ are continuous and increasing in $(0, \, \frac{1}{2}]$ and
$\lim_{x \to 0} G(x) = \lim_{x \to 0} H(x) = 0$, there is a measure $\mu$ on
$[-\frac{1}{2}, \, \frac{1}{2}]$ such that
	\begin{align*}
	&\mu([x, \, 0]) = G(|x|), \qquad -\frac{1}{2} \leq x < 0 \\
	&\mu([0, \, x]) = H(x), \qquad \phantom{-||}0 < x \leq \frac{1}{2}.
	\end{align*}
Now,
	\begin{align*}
	\frac{G(x)}{H(x)} &=
	\exp\left(\exp\left( \frac{1}{x + x^2} \right) - \exp\left( \frac{1}{x - x^2} \right)\right) \\
	&=
	\exp\left(\exp\left( \frac{1}{x} - \frac{1}{1 + x} \right) 
	- \exp\left( \frac{1}{x} + \frac{1}{1 - x} \right)\right) \\
	&=
	\exp\left(e^{\frac{1}{x}} \cdot \left( \exp\left( \frac{-1}{1 + x} \right)
	- \exp\left( \frac{1}{1 - x} \right) \right) \right) \to 0,
	\quad x \to 0,
	\end{align*}
since the inner parenthesis goes to $e^{-1} - e$ so that the argument of the
outer exponential function goes to $-\infty$.
For $c \in (0, \, 1)$,
	\begin{align*}
	\frac{H(cx)}{H(x)} &=
	\exp\left(\exp\left(\frac{1}{x + x^2}\right) - \exp\left(\frac{1}{cx + c^2x^2}\right)\right) \\
	&\leq
	\exp\left(\exp\left(\frac{1}{x + x^2}\right) - \exp\left(\frac{1}{c}\frac{1}{x + x^2}\right)\right) \\
	&=
	\exp\left(\exp\left(\frac{1}{x + x^2}\right) -
	\exp\left(\frac{1}{x + x^2}\right)^{1 / c}\right) \to 0, \quad x \to 0,
	\end{align*}
since $\exp(\frac{1}{x + x^2}) \to \infty$ and $\frac{1}{c} > 1$.
Hence for any $c \in (0, \, 1)$,
	$$
	\scfl{\mu}{0}{t}([-1, \, c]) = \frac{G(e^{-t}) + H(ce^{-t})}{G(e^{-t}) + H(e^{-t})}
	\to 0, \quad t \to \infty,
	$$
so $\lim_{t \to \infty} \scfl{\mu}{0}{t} = \delta_{1}$.

By the implicit function theorem, there is a diffeomorphism $g$ defined in a neighbourhood
$(-\varepsilon, \, \varepsilon)$ of $0$ such that $g(0) = 0$, $g'(0) = 1$ and
	$$
	g(x) + g(x)^2 = x - x^2.
	$$
Extend $g$ to a diffeomorphism in $\diffk{1}$ and let $f = g^{-1}$. Then for
$x \in (-\varepsilon, \, 0)$
	\begin{align*}
	G(|f^{-1}(x)|) &=
	\exp\left(-\exp\left(\frac{1}{-g(x) - g(x)^2}\right)\right) \\
	&=
	\exp\left(-\exp\left(\frac{1}{-x + x^2}\right)\right)
	=
	H(|x|),
	\end{align*}
and for $x \in (0, \, \varepsilon)$
	$$
	H(f^{-1}(x)) =
	\exp\left(-\exp\left(\frac{1}{g(x) + g(x)^2}\right)\right)
	=
	\exp\left(-\exp\left(\frac{1}{x - x^2}\right)\right)
	=
	G(x).
	$$
Thus $(f\mu)(A) = \mu(-A)$ for any measurable $A \subset (-\varepsilon, \, \varepsilon)$ and
it follows that $\lim_{t \to \infty} \scfl{(f\mu)}{0}{t} = \delta_{-1}$.
\end{example}

As Example \ref{example4} shows, the hypothesis in Proposition \ref{succeedproposition} is
not the weakest possible. Proposition \ref{succeedproposition} was proved by showing that if
$P(\{ (1 - \beta)\delta_{-1} + \beta \delta_1; \, 0 \leq \beta \leq 1 \}) = 0$
for every distribution $P$ in the $\omega$-limit set of $\scdi{\mu}{x}{T}$, then \eqref{TC}
is satisfied. The following proposition says that the opposite implication holds ---
thus to make the hypothesis of Proposition \ref{succeedproposition} weaker one would have to use
something other than Lemma \ref{TClemma}.

\begin{proposition}
Let $\mu \in \mathcal(\rea)$ and $x \in \supp \mu$ and suppose that the condition
\eqref{TC} is satisfied. Then 
$P(\{ (1 - \beta)\delta_{-1} + \beta \delta_1; \, 0 \leq \beta \leq 1 \}) = 0$
for every distribution $P$ in the $\omega$-limit set of $\scdi{\mu}{x}{T}$.

\begin{proof}
After a translation of the real line it may be assumed that $x = 0$. For any $\gamma > 1$
and any $K > 1$ let
	$$
	V_{\gamma}^K = \left\{ \nu \in \mathcal{M}^{\square}; \, 
	\nu\left( \frac{1}{\gamma} I \right)
	< \frac{1}{K} \right\};
	$$
this is an open set by Lemma \ref{opensetlemma} and it includes
$\{ (1 - \beta)\delta_{-1} + \beta \delta_1; \, 0 \leq \beta \leq 1 \}$. Let $\varepsilon > 0$.
By \eqref{TC} it is possible to choose $K$ and $\gamma > 1$ such that
	\begin{align*}
	\limsup_{T \to \infty} \scdi{\mu}{0}{T}(V_{\gamma}^K) &=
	\limsup_{T \to \infty} \frac{1}{T} \lambda\left(
	\left\{ t \leq T; \, \frac{\mu(\frac{1}{\gamma} I_t )}{\mu(I_t)} < \frac{1}{K} \right\}
	\right) \\
	&=
	\limsup_{T \to \infty} \frac{1}{T} \lambda\left(
	\left\{ t \leq T; \, \frac{\mu(I_t )}{\mu(\gamma I_t)} < \frac{1}{K} \right\}
	\right) \\
	&=
	\limsup_{T \to \infty} \frac{1}{T} \lambda\left(
	\left\{ t \leq T; \, \frac{\mu(\gamma I_t )}{\mu(I_t)} > K \right\}
	\right) < \varepsilon.
	\end{align*}
Now let $T_1, T_2, \ldots$ be a sequence increasing to infinity such that
$\scdi{\mu}{0}{T_k} \to P$. Then
	\begin{align*}
	P(\{ (1 - \beta)\delta_{-1} + \beta \delta_1; \, 0 \leq \beta \leq 1 \}) &\leq
	P(V_{\gamma}^K) \stackrel{*}{\leq}
	\liminf_{k \to \infty} \scdi{\mu}{0}{T_k} (V_{\gamma}^K) \\ &\leq
	\limsup_{T \to \infty} \scdi{\mu}{0}{T} (V_{\gamma}^K) < \varepsilon,
	\end{align*}
where the inequality at $*$ holds since $V_{\gamma}^K$ is open. The proposition
follows since $\varepsilon$ is arbitrary.
\end{proof}
\end{proposition}

\section{Another example}
In Example \ref{example3}, \ref{example4} and \ref{example5} above, the measure $\mu$ generates a distribution
$P$ that is a point mass on a convex combination of $\delta_{-1}$ and $\delta_{1}$, so there is certainly
$\nu \in \supp P$ with $0 \notin \supp \nu$. In all these examples this happens only at $x = 0$ ---
for any other point in the support of $\mu$, the generated distribution would be a point mass on
$\lambda$. This section will give an example of a measure $\mu$ that
generates the same distribution $P$ at $\mu$-a.e. $x$, where $P$ has $\frac{1}{2}(\delta_{-1} + \delta_1)$
in its support.

Let $\Lambda$ be a countable set with at least two elements and let $\{ [\,l\,] \}_{l \in \Lambda}$
be a collection of compact subintervals of $I$ whose interiors are pairwise disjoint. 
A uniformly expanding map $\sigma: \bigcup [\,l\,] \to I$ that maps the interior of each
$[\,l\,]$ linearly onto $(-1, \, 1)$ and the boundary of each $[\,l\,]$ into $\{ \pm 1\}$ will
be called a \emph{piecewise linear shift map}.
A \emph{cylinder} of generation $n$ is a compact interval of
the form
	$$
	[l_1 \ldots l_n] =
	\closure{\{x; \, \sigma^k(x) \in [l_{k - 1}] \text{ for } k = 1, \ldots, n \}}. 
	$$
By convention, $I$ is the cylinder $[\phantom{l}]$ of generation $0$.
Any pair of distinct cylinders of the same generation have disjoint interiors
and $[l_1 \ldots l_m] \subset [l'_1 \ldots l'_n]$ if and only if
$l_1 \ldots l_m$ is a prefix of $l'_1 \ldots l'_n$.
If $x$ lies in the interior of a cylinder of generation $n$, then the first
$n$ \emph{digits} $x_1, \ldots, x_n$ of $x$ are defined by $x \in [x_1 \ldots x_n]$.

Given an assignment of positive probabilities $(p_{l})_{l \in \Lambda}$,
the corresponding \emph{Bernoulli measure} is the unique probability measure
$\mu$ satisfying
	$$
	\mu([l_1 \ldots l_n]) = p_{l_1} \cdot \ldots \cdot p_{l_n}
	$$
for all cylinders $[l_1 \ldots l_n]$. The Bernoulli measure is
$\sigma$-invariant and ergodic, and $\supp \mu = \bigcap_{k = 0}^{\infty} \closure{\sigma^{-k}(I)}$
(this would be false if some $p_{l}$ were $0$).
For $\mu$-a.e. $x$, all digits $x_1, x_2, \ldots$ are defined.

\begin{lemma}	\label{selfsimlemma}
Let $\mu$ be a Bernoulli measure with respect to an orientation preserving piecewise linear
shift map $\sigma$. For $k = 0, 1, \ldots$ define the function $T_k: I \to [0, \, \infty]$
by
	$$
	T_k(x) =
	\begin{cases}
	\inf\{ t \geq 0; \, x + I_t \subset [x_1, \ldots, x_k] \}
	& \text{if } x_1, \ldots, x_k \text{ are defined} \\
	\infty
	& \text{otherwise},
	\end{cases}
	$$
and suppose that $T_1 - T_0$ is $\mu$-integrable. Then at $\mu$-a.e. $x$, $\mu$ generates
the distribution $P$ given by
	$$
	\int g \intd{P} =
	\frac{1}
	{\int (T_1 - T_0) \intd{\mu}} \int \int_{T_0(x)}^{T_1(x)} g(\scfl{\mu}{x}{t}) \intd{t} \intd{\mu(x)},
	\quad g \in C(\mathcal{M}^{\square}).
	$$

\begin{proof}
It can be that $T_1(x) = T_0(x)$ only if $x$ lies in the left half of the leftmost basic
interval or in the right half of the rightmost interval. Thus there is a basic
interval where $T_1 - T_0 > 0$ if $\Lambda$ has at least three elements, and if $\Lambda$
has two elements, say $\Lambda = \{ 0, \, 1 \}$, then $T_1 - T_0 > 0$ on at least one
of the cylinders $[01]$ and $[10]$. It follows that the denominator in the expression
for $P$ is positive.

Take any $g \in C(\mathcal{M}^{\square})$ and fix some
$\mu$-typical $x$ (meaning that the Birkhoff averages converge at $x$ to the integral with
respect to $\mu$) such that
$[x_1 \ldots x_k]$ is defined for all $k$. Since $\mu$ is
a Bernoulli measure, there are constants $c_k$ such that
	$$
	\sigma^k \left( \restr{\mu}{[x_1, \ldots, x_k]} \right) = c_k \mu
	$$
for all $k$, and thus
	\begin{equation}	\label{thesecondeq}
	\scfl{\mu}{x}{T_k(x)} =
	\left( \sigma^k \left( \restr{\mu}{[x_1, \ldots, x_k]} \right)\right)_{\sigma^k(x), T_0(\sigma^k(x))} =
	\scfl{\mu}{\sigma^k(x)}{T_0(\sigma^k(x))}.
	\end{equation}
Now,
	\begin{equation}	\label{thefirsteq}
	T_{k + 1}(x) - T_k(x) = T_1(\sigma^k(x)) - T_0(\sigma^k(x)),
	\end{equation}
so
	$$
	\lim_{n \to \infty} \frac{T_n(x)}{n} =
	\lim_{n \to \infty} \frac{1}{n} \left( T_0(x) +
	\sum_{k = 0}^{n - 1} T_{k + 1}(x) - T_k(x) \right)
	= \int (T_1 - T_0) \intd{\mu}
	$$
since $x$ is $\mu$-typical. In particular, this implies that if $n(T)$ is the unique natural
number such that $T_{n(T)} (x) \leq T < T_{n(T) + 1} (x)$ then
	$$
	\lim_{T \to \infty} \frac{T_{n(T)}(x)}{T} = 1.
	$$
Using this at the first step together with the fact that $g$ is bounded gives
	\begin{align*}
	&\lim_{T \to \infty} \frac{1}{T} \int_0^T g(\scfl{\mu}{x}{t}) \intd{t} =
	\lim_{T \to \infty} \frac{1}{T} \int_{T_0(x)}^{T_{n(T)} (x)} g(\scfl{\mu}{x}{t}) \intd{t} = \\
	&\lim_{n \to \infty} \frac{n}{T_n(x)} \frac{1}{n} \int_{T_0(x)}^{T_n(x)} g(\scfl{\mu}{x}{t}) \intd{t} = \\
	&\frac{1}{\int (T_1 - T_0) \intd{\mu}} \lim_{n \to \infty} \frac{1}{n} 
	\sum_{k = 0}^{n - 1} \int_{T_k(x)}^{T_{k + 1}(x)} g(\scfl{\mu}{x}{t}) \intd{t}
	= 
	\left[
	\begin{matrix}
	s = t - T_k(x) + T_0(\sigma^k(x)) \\
	\text{use \eqref{thesecondeq} and \eqref{thefirsteq}}
	\end{matrix}
	\right]
	= \\
	&\frac{1}{\int (T_1 - T_0) \intd{\mu}} \lim_{n \to \infty} \frac{1}{n} 
	\sum_{k = 0}^{n - 1} \int_{T_0(\sigma^k(x))}^{T_1(\sigma_k(x))}
	g(\scfl{\mu}{\sigma^k(x)}{s}) \intd{s}.
	\end{align*}
The lemma follows from Birkhoff's theorem applied to the last member, since
	$$
	x \mapsto \int_{T_0(x)}^{T_1(x)} g(\scfl{\mu}{x}{s}) \intd{s}
	\leq (T_1(x) - T_0(x)) \sup |g|
	$$
is $\mu$-integrable. 
\end{proof}
\end{lemma}

The support of the distribution $P$ in Lemma \ref{selfsimlemma} is \closure{E}, where
	$$
	E = \{ \scfl{\mu}{x}{t}; \, x \in \supp \mu, \, T_0(x) < t < T_1(x) \}.
	$$
If $\nu \notin \closure{E}$ then by Urysohn's lemma there is a non-negative function
$g \in C(\mathcal{M}^{\square})$ such that $g(\nu) > 0$ and $g$ is $0$ on $\closure{E}$. Then
	$$
	\int g \intd{P} =
	\frac{1}{\int (T_1 - T_0) \intd{\mu}} \int \int_{T_0(x)}^{T_1(x)} g(\scfl{\mu}{x}{t}) \intd{t} \intd{\mu(x)}
	= 0,
	$$
so the open neighbourhood $\{ g > 0 \}$ of $\nu$ has $P$-measure $0$. Thus $\supp P \subset \closure{E}$.
To see the other inclusion, the next lemma is needed.

\begin{lemma}	\label{contlemma}
Let $\mu \in \mathcal{M}(\rea)$ and fix $t_0 > 0$ and $x_0 \in \supp \mu$. Let
	$$
	\Delta = \{ (x, t); \, t \leq t_0, \, |x - x_0| \leq e^{-t} - e^{-t_0} \}.
	$$
Then
	$$
	\lim_{\substack{(x, t) \to (x_0, t_0)\\(x, t) \in \Delta}} \scfl{\mu}{x}{t} = \scfl{\mu}{x_0}{t_0}
	\qquad (\text{weak-}*).
	$$

\begin{proof}
Let $I_{x, t} = x + I_t$.
If $(x, t) \in \Delta$ then $I_{x_0, t_0} \subset I_{x, t}$, so
$\chi_{I_{x, t}} \to \chi_{I_{x_0, t_0}}$ pointwise when $(x, t) \to (x_0, t_0)$. Thus by the
dominated convergence theorem
	$$
	\lim_{\substack{(x, t) \to (x_0, t_0)\\(x, t) \in \Delta}}
	\mu(I_{x, t}) = \mu(I_{x_0, t_0}).
	$$
Take any $\varphi \in C([-1, 1])$. Again by the dominated convergence theorem,
	\begin{align*}
	\lim_{\substack{(x, t) \to (x_0, t_0)\\(x, t) \in \Delta}} \scfl{\mu}{x}{t} (\varphi)
	&= \lim_{\substack{(x, t) \to (x_0, t_0)\\(x, t) \in \Delta}}
	\frac{1}{\mu(I_{x, t})} \int \chi_{I_{x, t}}(\xi) \varphi(e^t (x - \xi)) \intd{\mu(\xi)} \\
	&=
	\frac{1}{\mu(I_{x_0, t_0})} \int \chi_{I_{x_0, t_0}}(\xi) \varphi(e^{t_0} (x_0 - \xi)) \intd{\mu(\xi)}
	= \scfl{\mu}{x_0}{t_0} (\varphi).
	\end{align*}
\end{proof}
\end{lemma}

For each $k$, the function $T_k$ is finite and continuous in the interior of
each cylinder of generation $k$, and if $x$ does not lie in the interior of
any such cylinder then $T_k(x) = \infty$ and $\lim_{y \to x}T_k(y) = \infty$ ---
thus $T_k$ is continuous at every $x \in I$ in the sense that
$\lim_{y \to x} T_k(y) = T_k(x)$. Take any $\scfl{\mu}{x_0}{t_0} \in E$ and let $V$
be an open set around $\scfl{\mu}{x_0}{t_0}$. Because of Lemma \ref{contlemma}
and the continuity of $T_0$ and $T_1$, there is an open rectangle
$X \times Y$ with $x_0 \in X$ such that
	$$
	\begin{array}{ll}
	\{ \scfl{\mu}{x}{t}; \, (x, t) \in X \times Y \} \subset V 	& \text{and } \\
	T_0(x) < t < T_1(x)										& \text{for all} (x, t) \in X \times Y.
	\end{array}
	$$
Then $P(V) \geq (\mu \times \lambda) (X \times Y) > 0$
since $x_0 \in \supp \mu$. Thus
$E \subset \supp P$ and since $\supp P$ is closed it follows that $\closure{E} \subset \supp P$.

\begin{example}
Let $\Lambda = \ints \setminus\{ 0 \}$ and for $l \in \Lambda$ define $[\,l\,]$ to
be the interval
	$$
	[\,l\,] = \sign(l) [2^{-|l|}, \, 2^{-(|l| - 1)}].
	$$
Let $\sigma: I \setminus \{ 0 \} \to I$ be a map that maps the interior
of each $[\,l\,]$ linearly onto $(-1, \, 1)$, preserving orientation. Let $p_1, p_2, \ldots$ be
a sequence of positive numbers that sum to $1 / 2$ and satisfy
	\begin{equation}	\label{decrcond}
	\lim_{n \to \infty} \frac{\sum_{l = n + 1}^{\infty} p_l }{p_n} = 0,
	\end{equation}
set $p_{l} = p_{|l|}$ for $l < 0$ and let $\mu$ be the corresponding Bernoulli measure.
Note that $0 \in \supp \mu$ since all $p_l$ are positive.

Define $T_0$ and $T_1$ as in Lemma \ref{selfsimlemma}. Since
	$$
	T_0(x) = -\log(1 - |x|)
	$$
and $\mu$ is symmetric around $0$,
	\begin{align*}
	\int T_0 \intd{\mu} &=
	-2 \int_{[0, 1]} \log(1 - x) \intd{\mu(x)} \leq
	-2 \sum_{k = 0}^{\infty} \mu([1 - 2^{-k}, \, 1 - 2^{-(k + 1)}]) \log(2^{-(k + 1)}) \\
	&=
	2 (1 - p_1) \log 2 \sum_{k = 0}^{\infty} (k + 1) p_1^{k} =
	\frac{2 \log 2}{(1 - p_1)} <
	\infty.
	\end{align*}
Thus $T_0$ is integrable. Furthermore
	$$
	T_1(x) = \log\left( \frac{2}{\text{length of } [x_1]} \right) + T_0(\sigma(x))
	= (|x_1| + 1) \log 2 + T_0(\sigma(x)),
	$$
where $x_1$ is the integer such that $x \in [x_1]$, so
	$$
	\int T_1 \intd{\mu} = 2 \log 2 \sum_{l = 1}^{\infty} (l + 1) p_l + \int T_0 \intd{\mu}
	$$
using that $\mu$ is $\sigma$-invariant. The sum converges because of \eqref{decrcond}, so
$T_1$ is integrable as well. Thus by Lemma \ref{selfsimlemma}, $\mu$ generates a distribution
$P$ at $\mu$-a.e. $x$ with
	$$
	\supp P = \closure{
	\{ \scfl{\mu}{x}{t}; \, x \in \supp \mu, \, T_0(x) \leq t \leq T_1(x) \}
	}.
	$$

Now let
	$$
	t_{m, n} = -\log\left( (1 + 4^{-m}) 2^{-n} \right)
	$$
and consider the measures $\scfl{\mu}{0}{t_{m, n}}$. Since $0 \in \supp \mu$ and
$T_0(0) = 0$ and $T_1(0) = \infty$, they are all in $\supp P$. Moreover,
	\begin{align*}
	\scfl{\mu}{0}{t_{m, n}} \left( \pm \left[ 1 - \frac{1}{1 + 4^m}, \, 1\right]\right)
	&=
	\frac	{\mu ( \pm [2^{-n}, \, (1 + 4^{-m}) 2^{-n} ] )}
			{\mu ([-(1 + 4^{-m}) 2^{-n}, \, (1 + 4^{-m}) 2^{-n}] )} \\
	&=
	\frac{p_1^m p_n}{2(p_1^m p_n + \sum_{l = n + 1}^{\infty} p_l)}
	\to \frac{1}{2},
	\quad n \to \infty
	\end{align*}
by \eqref{decrcond}. Thus for any $\varepsilon > 0$ there is a measure in $\supp P$ that
gives measure $> \frac{1}{2} - \varepsilon$ to each of the intervals $\pm[1 - \varepsilon, \, 1]$.
Since $\supp P$ is closed, it follows that $\frac{1}{2}(\delta_{-1} + \delta_1) \in \supp P$.

For $\mu$-a.e. $x$, the local dimension of $\mu$ at $x$ is
	$$
	D(\mu, x) = \frac{h_{\mu}(x)}{u(x)}
	= \frac{-1}{\log 2} \frac{\sum_{l = 1}^{\infty} p_l \log p_l}{\sum_{l = 1}^{\infty} (l + 1)p_l},
	$$
where $h_{\mu}(x)$ is the local entropy and $u(x)$ is the Lyapunov exponent at $x$. By
suitable choice of the probabilities $(p_l)$, the dimension can be made to take any
value in $(0, 1)$. Consider for example
	$$
	p^N_l =
	\begin{cases}
	c_N 2^{-l}	& \text{if } l \leq N \\
	c_N \frac{1}{l!}	& \text{if } l > N,
	\end{cases}
	$$
where the constant $c_N$ is chosen so that $(p^N_l)$ sums to $1 / 2$. These probabilities
satisfy \eqref{decrcond} for any finite $N$, and
	$$
	\frac{-1}{\log 2}
	\frac{\sum_{l = 1}^{\infty} p^N_l \log p^N_l}{\sum_{l = 1}^{\infty} (l + 1)p^N_l} \geq
	\frac{\sum_{l = 1}^{N} 2^{-(l + 1)} (l + 1)}{\sum_{l = 1}^{\infty} 2^{-(l + 1)} (l + 1)}
	\to 1, \quad N \to \infty
	$$
(this is not surprising since $\mu$ becomes Lebesgue measure if $N = \infty$). Thus, the
dimension can be made arbitrarily close to $1$. On the other hand, if two probabilities
are swapped then the Lyapunov exponent changes but the entropy does not. In particular,
it is possible to make the local dimension arbitrarily close to $0$ by putting the largest
probability at a high index.
Thus for any $\varepsilon > 0$, there are $(p^{\varepsilon}_l)$ and $(q^{\varepsilon}_l)$ that
sum to $1 / 2$, satisfy \eqref{decrcond} and give $\mu$ dimension at least $1 - \varepsilon$ and at
most $\varepsilon$ respectively. Any convex combination of $(p^{\varepsilon}_l)$ and $(q^{\varepsilon}_l)$
also sums to $1 / 2$ and satisfies \eqref{decrcond}, and since the dimension varies continuously with the
probabilities, it takes all values in $(\varepsilon, \, 1 - \varepsilon)$ on the line segment
between $(p^{\varepsilon}_l)$ and $(q^{\varepsilon}_l)$.
\end{example}

\bibliographystyle{plain}
\bibliography{references}
\end{document}